\documentclass[epj]{svjour}

\usepackage{amsfonts}
\usepackage{amssymb,latexsym}
\usepackage{amsmath}
\usepackage{enumerate}
\usepackage[small,bf]{caption}
\usepackage{graphics}
\usepackage{graphicx}

\numberwithin{equation}{section}

\def\bfR{\mathbb{R}}

\def\mcH{\mathcal{H}}
\def\mcL{\mathcal{L}}
\def\mcF{\mathcal{F}}

\begin{document}

\title{Instability conditions for some periodic BGK waves in the Vlasov-Poisson system}

\author{Stephen Pankavich\inst{1}
\thanks{\emph{Department of Applied Mathematics and Statistics, Colorado School of Mines, 1500 Illinois St., Golden, CO 80401}
This work was supported by the Center for Undergraduate Research in Mathematics subcontracted under NSF Grant DMS-0636648 and independently by the National Science Foundation under NSF Grants DMS-0908413 and DMS-1211667.}
\and Robert Allen\inst{2}
%
}                     
\offprints{}          
\institute{Department of Applied Mathematics and Statistics, Colorado School of Mines, 1500 Illinois St., Golden, CO 80401
\and Space Science Center and Department of Physics, 8 College Rd., University of New Hampshire, Durham, NH 03824}

\date{Received: February 28, 2014 / Revised version: September 14, 2014}
%

\abstract{A one-dimensional, collisionless plasma given by the Vlasov-Poisson system is considered and the stability properties of periodic steady state solutions known as Bernstein-Greene-Kruskal (BGK) waves are investigated. Sufficient conditions are determined under which BGK waves are linearly unstable under perturbations that share the same period as the equilibria. It is also shown that such solutions cannot support a monotonically decreasing particle distribution function.
} 
%


%
%






\maketitle

\markboth{S. PANKAVICH AND R. ALLEN}{INSTABILITY OF SOME BGK WAVES}

\section{Introduction}

A plasma is a partially or completely ionized gas, in which the motion of unbound ions generates electric and magnetic fields that strongly affect individual particle motion. When a plasma is of low density or the time scales of interest are sufficiently small, it is deemed to be ``collisionless'', as collisions between particles become infrequent.  Many examples of collisionless plasmas occur in nature, including the solar wind, galactic nebulae, and comet tails.

One fundamental, one-dimensional model of collisionless plasma dynamics is given by a system of partial differential equations known as the Vlasov-Poisson system:
\begin{equation} \tag{VP} \label{VP} \left. \begin{gathered}
\partial_t f + v\partial_x f - E \partial_v f = 0\\
\partial_x E = 1 - \int f dv. \end{gathered} \right \}
\end{equation}
Here $f =f(t,x,v)$ represents the distribution of electrons in the plasma, while $E= E(t,x)$ is the self-consistent electric field generated by ions and electrons. The independent variables, $t > 0$ and $x,v \in \bfR$ represent time, position, and velocity, respectively.  Instead of studying a large collection of ionic species interacting with the electrons, the density of ions is given by a neutralizing background, normalized to $1$ in the equations above. 

Of course, the study of plasmas has been an extremely rich area within both the physics and applied mathematics communities for decades. Early work regarding the stability and instability of plasmas includes \cite{BDM}, \cite{Fowler}, \cite{LewisSymon}, \cite{Pfirsch}, \cite{SPW}, and \cite{Schwarz}.  More recent studies of one-dimensional plasma dynamics include \cite{Balmforth}, \cite{BP}, \cite{Lin}, \cite{LinNonlinear}, and \cite{NP} while for general references regarding the Vlasov equation and models in kinetic theory, such as (\ref{VP}) and its electromagnetic counterpart, the Vlasov-Maxwell system, we mention \cite{Glassey}, \cite{MorrisonTTSP}, and \cite{VKF}. 

In the current study we are interested in the stability properties of time-independent solutions of the system.  More specifically, we wish to study well-known steady states called BGK waves.  In 1957 Bernstein, Greene, and Kruskal \cite{BGK} showed the existence of an infinite family of exact solutions to (\ref{VP}) that have come to be known as BGK waves.  Since then, the stability or instability of these solutions has been of great interest to the plasma physics and applied mathematics communities (\cite{GuoStrauss95}, \cite{GuoStrauss98}, \cite{GuoStrauss99}, \cite{ManfrediBertrand}, \cite{LinNonlinear}).  In particular, recent experimental identifications of electrostatic solitary waves in space plasmas have provided further justification for the study the properties of such steady states. A BGK wave is a steady solution of the form
\begin{equation}
\label{BGK}
\left. \begin{gathered} \mathring{f}(x,v) = \mu \left (\frac{1}{2} v^2 + \phi(x) \right )\\ \mathring{E}(x) = \phi'(x)\end{gathered} \right \}
\end{equation}
where $\phi$ is a periodic solution of
\begin{equation}
\label{phi}
\phi''(x) = 1- \int \mu \left (\frac{1}{2} v^2 + \phi(x) \right ) \ dv.
\end{equation}
In \cite{Lin} the stability properties of such solutions to (\ref{VP}) were studied. There it was determined that any BGK wave is linearly unstable with respect to \emph{multi}-periodic perturbations, that is, perturbations whose periods are integer multiples ($2$ or greater) of the period of the steady state $\phi$. Unfortunately, such a general result in the singly-periodic case, in which the perturbation and BGK wave possess the same period, remains unknown and does not seem likely.  As the author states in \cite{Lin}, one cannot hope that all periodic BGK waves are unstable with respect to such perturbations. Instead, the stability properties may depend delicately on the specific BGK wave of interest. To date little information regarding this question is known, though some results concerning weakly-inhomogeneous solutions have been obtained (\cite{GuoStrauss99}, \cite{ManfrediBertrand}).  The main focus of our paper, then, is to determine sufficient conditions under which these BGK waves $(\mu, \phi)$ are linearly unstable due to \emph{singly}-periodic perturbations.

As Lin constructed a strong framework for studying the problem, we will follow many of the assumptions and previous results developed in \cite{Lin}. To be precise, we make the following assumptions on the BGK waves (\ref{BGK}) - (\ref{phi}) throughout:
\begin{enumerate}[(i)]
\item \label{one} $\mu \in C^1(\bfR)$ is nonnegative.
\item \label{two} $\mu$ satisfies the condition of neutrality, i.e. $$ \int \mu \left ( \frac{1}{2} v^2 \right ) \ dv = 1.$$
\item \label{three} $\mu'$ decays at infinity, i.e. there is $\gamma > 1$ and $C > 0$ such that $$\vert \mu'(y) \vert \leq \frac{C}{1 + \vert y \vert^\gamma}.$$

\item \label{four} We arrange the period of $\phi$ in a specific manner.  Let $P_\phi$ be the minimal period of the solution $\phi$ from (\ref{phi}) so that $$ \phi(x) = \phi(x + P_\phi) $$ for every $x \in \mathbb{R}$ and define the quantities
$$ \begin{gathered}\phi_- = \min_{x \in [0,P_\phi]} \phi(x)\\ \phi_+ = \max_{x \in [0,P_\phi]} \phi(x). \end{gathered}$$
  Then, without loss of generality we may rearrange the starting point to impose conditions on the values of $\phi$ so that it satisfies:
$$\begin{gathered} \phi(0) = \phi(P_\phi) = \phi_+,\\
\phi \left (\frac{P_\phi}{2} \right) = \phi_-,\\
\phi(x) = \phi(P_\phi - x), \quad \forall x \in [0,P_\phi]. \end{gathered} $$  Additionally, without loss of generality we can take $\phi$ to be strictly decreasing on the interval $\left [0,\frac{P_\phi}{2} \right]$.
\end{enumerate}

With this structure in place, we will show that (\ref{phi}) cannot possess solutions for which $\mu$ is strictly decreasing.  
More importantly, we will show in subsequent sections that, while one cannot hope that all such solutions to (\ref{BGK})-(\ref{phi}) are linearly unstable with respect to $P_\phi$-periodic perturbations, conditions do exist which guarantee their instability at the linear level.  To be more precise further definitions are needed.  For $\mu$ given above we define the function $$q(x) = \int \mu'\left (\frac{1}{2}v^2 + \phi(x) \right )\ dv$$ and let $\lambda_0 < 0$ be the smallest eigenvalue, with corresponding first eigenfunction $\psi_0 \in H^2(0,P_\phi)$, of the Sturm-Liouville \cite{Zettl} problem
$$\psi''(x) + (q(x) + \lambda) \psi(x) = 0$$ with $\psi(0) = \psi(P_\phi) = 0$.  We will prove that if $q$ satisfies
\begin{equation}
\tag{v} \label{six}
\int_0^{\frac{1}{2}P_\phi} (q(x) + \lambda_0) (q(x) + \frac{4}{3}\lambda_0) \vert \psi_0(x) \vert^4 \ dx < 0
\end{equation}
the BGK wave $(\mu, \phi)$ must be linearly unstable with respect to perturbations of period $P_\phi$. 
In particular, this demonstrates that for a BGK wave to be stable, the values of $q$ cannot be essentially localized within the interval $[-\lambda_0, -\frac{4}{3} \lambda_0]$. Our main results will be stated precisely in the next section, after deriving the specific partial differential equations that a perturbation must satisfy. 

\section{Linearized Perturbation Equations and Main Results}

To begin this section, we first derive the system of PDEs that perturbations of (\ref{VP}) must satisfy so that we may study their behavior.  To determine the stability properties of these solutions we consider the initial value problem with time dependent perturbations of the form:
\begin{eqnarray*}
f(t,x,v) &=& \mu\left (\frac{1}{2}v^2 + \phi(x) \right) + F(t,x,v)\\
E(t,x) &=& \phi'(x) + \beta(t,x).
\end{eqnarray*}
Using these quantities in (\ref{VP}) yields equations for the perturbations, namely

$$\partial_t F + v \phi' \mu' + v \partial_x F - (\phi' + \beta) v\mu' - (\phi' + \beta) \partial_v F = 0$$
$$  \phi''  + \partial_x \beta= 1 - \int \mu(e) dv - \int F(t,x,v) dv.$$
Since $(\mu,\phi)$ is a time-independent solution to (\ref{VP}), these equations simplify and we arrive at

\begin{equation}
\label{NL}
\left. \begin{gathered}
\partial_t F + v \partial_x F -\phi' \partial_v F =  v\beta\mu' + \beta \partial_v F\\
\partial_x \beta = - \int F dv.
\end{gathered}  \right \} \end{equation}
Notice that the system of partial differential equations (\ref{NL}) is nonlinear due to the appearance of the $\beta \partial_v F$ term.  We focus on studying the linearized system, so we remove this term, which yields

\begin{equation}
\label{Lin}
\left. \begin{gathered}
\partial_t F + v \partial_x F -\phi' \partial_v F =  v\beta\mu'\\
\partial_x \beta = - \int F dv.
\end{gathered}  \right \} \end{equation}
for unknown perturbations of both the distribution function $F(t,x,v)$ and the field $\beta(t,x)$.  This linearized system of PDE then constitutes the equations under study.
Thus, the main question of interest is whether solutions $(F, \beta)$ tend to zero over long times, thereby leading the solution $(f, E)$ of (\ref{VP}) to tend to the BGK waves, or whether solutions to (\ref{Lin}) may remain large over time, thereby causing $f$ and $E$ to stay far from equilibrium. 
Our first result shows that exponentially growing and decaying modes of \eqref{Lin} must occur in pairs of opposite sign.
\begin{theorem}
\label{Teig}
Let $\lambda \in \mathbb{R}$ be given  Then, equation \eqref{Lin} possesses a solution of the form
$$\begin{gathered} F(t,x,v) = e^{\lambda t} \mcF(x,v) \\
\beta(t,x) = e^{\lambda t} \Psi(x)  \end{gathered}$$ 
if and only if it possesses a solution of the form
$$\begin{gathered} F(t,x,v) = e^{-\lambda t} \mcF(x,-v) \\
\beta(t,x) = e^{-\lambda t} \Psi(x).
\end{gathered}$$  
\end{theorem}
\begin{proof}[Theorem \ref{Teig}]
The result of the theorem follows from the time-reversible Hamiltonian nature of \eqref{Lin} (see \cite{HagstromMorrison} for a more general result concerning the Vlasov equation).
In this case, we merely utilize the specific forms of solutions to the linearized equation.  Namely, using 
$$\begin{gathered} F(t,x,v) = e^{\lambda t} \mcF(x,v) \\
\beta(t,x) = e^{\lambda t} \Psi(x)  \end{gathered}$$ within (\ref{Lin}), we see that this is a solution if and only if
\begin{equation}
\label{mcFsystem}
\left. 
\begin{gathered}
\lambda \mcF(x,v) + v \partial_x \mcF(x,v) - \phi' \partial_v \mcF(x,v) = -\psi' v \mu'\\
\psi''(x) = \int \mcF(x,v) \ dv.
\end{gathered} \right \}
\end{equation}
Next, we instead impose 
$$\begin{gathered} F(t,x,v) = e^{-\lambda t} \mcF(x,-v) \\
\beta(t,x) = e^{-\lambda t} \Psi(x)
\end{gathered}$$  
within \eqref{Lin} and find
$$\begin{gathered}
-\lambda \mcF(x,-v) + v \partial_x \mcF(x,-v) + \phi' \partial_v \mcF(x,-v) = -\psi' v \mu'\\
\psi''(x) = \int \mcF(x,-v) \ dv.
\end{gathered}$$
Then, writing $w = -v$ in the first equation and removing a negative sign in front of each term, we find
$$\lambda \mcF(x,w) + w \partial_x \mcF(x,w) - \phi' \partial_v \mcF(x,w) = -\psi' w \mu'$$ 
which is identical to the first equation of \eqref{mcFsystem} under the relabeling $w=v$.
Changing variables using $w = -v$ within the integral of the second equation then yields the second equation of \eqref{mcFsystem}.  
Therefore, the two solutions satisfy the same system of equations, and the proof is complete.
\end{proof}

Therefore, one cannot expect asymptotic stability of solutions as all decaying modes must give rise to a growing mode.  We note, however, that Theorem \ref{Teig} does not guarantee the existence of either solution, and this is the focus of the remainder of the paper.
Our main result concerning the linear instability of periodic BGK waves (\ref{BGK}) - (\ref{phi}), is stated precisely in the following theorem.
\begin{theorem}
\label{T1}
Let $\mathring{f}$ and $\mathring{E}$ be the periodic BGK solutions of (\ref{VP}) described by (\ref{BGK}) with period $P_\phi$.  Assuming conditions {\normalfont (\ref{one}) - (\ref{six})}, there is $\lambda > 0$ and a solution of the form
$$\begin{gathered} F(t,x,v) = e^{\lambda t} \mcF(x,v) \\
\beta(t,x) = e^{\lambda t} \Psi(x)  \end{gathered}$$ to the linearized perturbation equations (\ref{Lin})
where $\mcF(\cdot, v)$ is a $P_\phi$-periodic function for every $v \in \bfR$, and $\Psi = -\psi'$ with $\psi \in H^2(0,P_\phi)$.
\end{theorem}
Hence, there exists a growing mode for the linearized perturbation equations with period $P_\phi$, and from Theorem \ref{Teig} one can construct a corresponding decaying mode from this solution, as well. 
Many authors have investigated questions in this vein under specific conditions on the function $q$, for instance, see \cite{AshBen} and \cite{Lavine} for results concerning single-well and convex functions, respectively. In the case of Theorem \ref{T1}, the properties of $q$ greatly depend on the local behavior of the distribution function $\mu$ and potential $\phi$, and even for different single-well or convex profiles, assumption (\ref{six}) may be satisfied or fail to hold. 
Regarding the proof, the main idea hinges on utilizing an equivalent formulation of the problem developed in \cite{Lin}.
Prior to stating the lemma containing this result, we first construct the characteristic curves for the particles in (\ref{Lin}).  Hence, define the curves $X(s,x,v)$ and $V(s,x,v)$, which we shall often abbreviate as $X(s)$ and $V(s)$ respectively, as solutions of the system of ordinary differential equations
\begin{equation}
\label{char}
\left.\begin{gathered}
\frac{\partial X}{\partial s}=V(s,x,v), \\
X(0,x,v)=x, \\
\frac{\partial V}{\partial s}=-\phi'(X(s,x,v)), \\
V(0,x,v)=v.
\end{gathered} \right \}
\end{equation}
From the definition of these curves, we can immediately construct an invariant of the physical system, namely the particle energy $$e = \frac{1}{2} v^2 + \phi(x) = \frac{1}{2} V(s)^2 + \phi(X(s))$$ which is independent of the time variable $s$.  This is obtained by multiplying the equation for $\frac{\partial V}{\partial s}$ in (\ref{char}) by $V(s)$ and expressing the remaining terms as a total derivative in $s$.  Throughout, we will denote both the particle energy and the exponential function by $e$, leaving the reader to differentiate between them due to context.  For instance, we will often write $\mu(e)$ in the future instead of $\mu \left ( \frac{1}{2} v^2 + \phi(x) \right )$ or $\mu \left ( \frac{1}{2} V(s)^2 + \phi(X(s)) \right )$ even though the three terms are equivalent. With this in hand, we may now summarize the previous work of \cite{Lin} in the context of singly-periodic perturbations of BGK waves.
\begin{lemma}[\cite{Lin}]
\label{L1}
Let $\lambda > 0$ be given.  There exists a nontrivial solution to (\ref{Lin}) of the form $$\begin{gathered} F(t,x,v) = e^{\lambda t} \mcF(x,v)\\ \beta(t,x) = e^{\lambda t} \Psi(x) \end{gathered}$$ 
if the functional $\mcL[\psi]$ defined by
\begin{eqnarray*}
\mcL[\psi] & := & \int_0^{P_\phi} \biggl [ \vert \psi'(x) \vert^2 - \int \mu'(e) \ dv \vert \psi(x) \vert^2 \biggr ] \ dx\\
& & \quad + \int_{\phi_+}^\infty \mu'(e) \frac{1}{P_f} \biggl ( \int_0^{P_\phi} \frac{\psi(y)}{\sqrt{2(e - \phi(y))}} \ dy \biggr )^2 \ de\\
& & \quad + 2 \int_{\phi_-}^{\phi_+} \mu'(e) \frac{1}{P_t} \biggl ( \int_{\alpha}^{P_\phi - \alpha} \frac{\psi(y)}{\sqrt{2(e - \phi(y))}} \ dy \biggr )^2 \ de.
\end{eqnarray*}
satisfies $\mcL[\psi] < 0$ for some $\psi \in H^1(0,P_\phi)$.
Here, $\alpha \in \left [ 0,\frac{1}{2}P_\phi \right] $ is the unique point in the interval such that $\phi(\alpha) = e$, for a given $e \in [\phi_-,\phi_+]$, while $$\begin{gathered} P_f = \int_{0}^{P_\phi} \frac{1}{\sqrt{2(e-\phi(y))}} dy,\\ P_t = \int_{\alpha}^{P_\phi - \alpha} \frac{1}{\sqrt{2(e-\phi(y))}} dy. \end{gathered}$$
Additionally, $\mcF$ and $\Psi$ are given by
\begin{equation}
\label{F}
\mcF(x,v) = \mu'(e) \left ( \int_{-\infty}^{0} \lambda e^{\lambda s} \psi \bigl ( X(s) \bigr ) ds -\psi(x) \right )
\end{equation}
and $\Psi = -\psi'$, where
where $X(s)$ is the spatial solution of the characteristic equations (\ref{char}).  
\end{lemma}

\begin{figure*}[t]
\centering \includegraphics[scale=1]{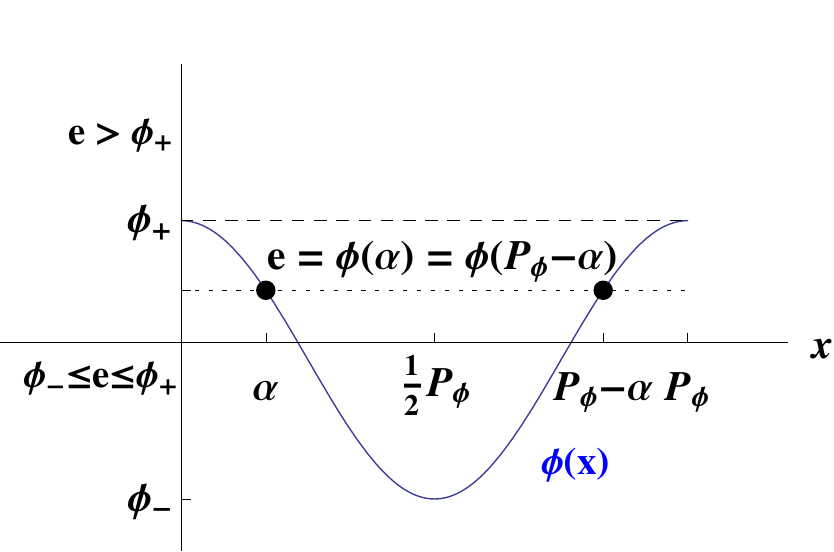}
\caption{A representative graph of the potential $\phi(x)$ for the BGK waves of (\ref{BGK})-(\ref{phi}).  The particles are divided into free particles (with period $P_f$), whose energy lies completely above the dashed line ($e > \phi_+$) and hence above the graph of $\phi$, and trapped particles (with period $P_t$), whose energy lies on or below the dashed line ($\phi_- \leq e \leq \phi_+$) and intersects the graph of $\phi$. For the latter, the symmetry of $\phi$ implies that the dotted line denoting the particle energy must intersect the graph of $\phi$ in exactly two points within the interval $[0,P_\phi]$. One of these points (denoted by $\alpha$) must be less than $\frac{1}{2} P_\phi$ and the other, at the reflected point $P_\phi - \alpha$, must be greater than $\frac{1}{2} P_\phi$. } \label{potential}
\end{figure*}

\noindent We refer the reader to \cite{Lin} for the details of the proof, but we can also briefly sketch the main ideas.  More specifically, utilizing the form of a growing mode solution $(F,\beta)$ from Theorem \ref{T1} within the linearized perturbation equations, it is a straightforward computation to show that \eqref{Lin} has a solution if and only if $A_\lambda \psi =0$ for an associated family of Hermitian dispersion operators $A_\lambda$.  With this, \cite{Lin} proves the continuity of $A_\lambda$ with respect to $\lambda > 0$ and positivity of $A_\lambda$ as $\lambda \to +\infty$, while the condition of Lemma \ref{L1}, namely $\mcL[\psi] < 0$, guarantees that $A_\lambda$ is negative as $\lambda \to 0^+$.  Hence, an application of the Intermediate Value Theorem then implies the existence of a value of $\lambda > 0$ for which $A_\lambda \psi = 0$.
To keep technical details to a minimum, we will omit a complete description of the dispersion operator $A_\lambda$ and focus instead on proving the condition $\mcL[\psi] < 0$.

Figure \ref{potential} contains a representation of the periodic wave $\phi$ and displays the associated periods of the particles occurring within $\mcL[\psi]$. In view of Lemma \ref{L1}, we can now see that the sign of $\mu'(e)$ may play a large role in the stability of solutions.  For instance, if $\mu'(e) \leq 0$ for $e \in [\phi_-,\infty)$ one could guarantee the last two expressions in $\mcL[\psi]$ remain non-positive while choosing $\psi$ to minimize the first integral. However, as we will show, BGK waves cannot support such distribution functions, meaning $\mu'(e)$ cannot remain nonpositive for all energies $e \in [\phi_-, \infty)$.

\begin{theorem}
\label{T2}
Assume conditions (\ref{one})-(\ref{four}) hold and the BGK wave $(\mu,\phi)$ satisfies (\ref{phi}), then $$q(x) = \int \mu'\left ( \frac{1}{2}v^2 + \phi(x) \right ) dv > 0$$ for some $x \in [0,P_\phi].$  In particular, $\mu$ cannot satisfy $\mu'(e) \leq 0$ for all $e \in [\phi_-,\infty)$.
\end{theorem}
We remark that the second statement of the the theorem regarding the non-monotonicity of $\mu(e)$ has been shown before within \cite{EFM} and \cite{VOK}.
Hence, the signs of the complicated terms in the representation of Lemma \ref{L1} depend very delicately on the behavior of $\mu'(e)$.  Regardless, we can eliminate this difficulty by utilizing the BGK waves and their symmetry to prove Theorem \ref{T1}. 

\begin{proof}[Theorem \ref{T1}]
Since the latter terms in the representation of $\mcL[\psi]$ are complicated, we will construct a specific choice of $\psi \in H^1(0,P_\phi)$ for which the first piece is negative and use symmetry to eliminate the contribution of the last two expressions.  To begin, recall (\ref{phi}):
$$\phi''(x) = 1 - \int \mu(e) dv.$$
We take an $x$-derivative to find
$$\phi'''(x) = -\left ( \int \mu'(e) \ dv \right ) \phi'(x)$$
or 
$$\phi'''(x) +  q(x) \phi'(x) = 0.$$
Hence, the function $u = \phi'$ satisfies the Sturm-Liouville problem
$$ \left . \begin{gathered}u''(x) + (q(x) + \lambda) u(x) = 0\\
u(0) = u(P_\phi) = 0
\end{gathered} \right. $$ for $u \in H^2(0,P_\phi)$ with corresponding eigenvalue $\lambda = 0$.  However, $\phi$ attains a maximum at $\frac{1}{2}P_\phi$ and possesses no other critical points on $(0,P_\phi)$.  Thus, $\phi'(\frac{1}{2}P_\phi) = 0$ and because $\phi'$ has exactly one root in $(0,P_\phi)$, it must be the second eigenfunction \cite{Zettl}. Therefore, there exists a first eigenfunction $u_0 \in \mcH$ and eigenvalue $\lambda_0 < 0$ such that
\begin{equation}
\label{psi0}
u_0''(x) + (q(x) + \lambda_0) u_0(x) = 0
\end{equation}
and $u_0(0) = u_0(P_\phi) = 0$.  Using the symmetry of $q$, namely $q(P_\phi - x) = q(x)$ for every $x \in [0,P_\phi]$, we see that the function $\tilde{u}_0(x) := u_0(P_\phi - x)$ is also a solution of the Sturm-Liouville problem with the same eigenvalue.  Due to uniqueness of solutions, it follows that there is $C \in \mathbb{R}$ such that $u_0(x) = Cu_0(P_\phi - x)$.  Evaluating this at the point $x = \frac{1}{2}P_\phi$, we see that either $C=1$ or $u_0(\frac{1}{2}P_\phi) = 0$.  As $u_0$ is the first eigenfunction, it has no roots in $(0,P_\phi)$.  Thus, $C = 1$ and  $u_0(x) = u_0(P_\phi-x)$.  Since $u_0 \in H^2(0,P_\phi)$ we may take a derivative of (\ref{psi0}) and find that $u_0'$ satisfies the symmetry condition $u_0'(P_\phi - x) = - u_0'(x)$ for every $x \in [0,P_\phi]$.  With these functions in place, define $\psi(x) = u_0(x) u_0'(x)$ for $x \in [0,P_\phi]$.  Using the aforementioned properties of $u_0$ and $u_0'$, we see that $\psi$ satisfies 
$$\psi(0) =\psi(P_\phi) = 0$$
with $\psi(P_\phi - x) = - \psi(x)$ for all $x \in [0,P_\phi]$. In addition, $\psi$ can be extended to the whole space in a smooth manner by imposing $P_\phi$-periodicity since
\begin{eqnarray*} \
\psi'(P_\phi) & = & \vert u_0'(P_\phi) \vert^2 + u_0(P_\phi) u_0''(P_\phi) \\
& = & \vert u_0'(P_\phi) \vert^2 \\
& = & \vert u_0'(0) \vert^2 \\
& = & \vert u_0'(0) \vert^2 + u_0(0) u_0''(0)\\
& = & \psi'(0).
\end{eqnarray*}
As $u_0 \in H^2(0,P_\phi)$, we see that $\psi \in H^1(0,P_\phi)$ and is smooth at the endpoints. Furthermore, this function satisfies
\begin{eqnarray*}
\vert \psi'\vert^2 - q \vert \psi \vert^2 & = & \vert (u_0')^2 + u_0u_0''  \vert^2 - q\vert u_0 u_0' \vert^2 \\
& = & \left \vert (u_0')^2 - (q + \lambda_0) (u_0)^2  \right \vert^2 - q\vert u_0 u_0' \vert^2 \\
& = & \vert u_0' \vert^4 - 2(q + \lambda_0) \vert u_0 u_0' \vert^2\\
& \ & \  + (q + \lambda_0)^2 \vert u_0 \vert^4  - q\vert u_0 u_0' \vert^2 \\
& = & \vert u_0' \vert^4 - 2\lambda_0(u_0  u_0')^2 + (q + \lambda_0)^2 \vert u_0 \vert^4\\
& \ & \  - 3q\vert u_0 u_0' \vert^2
\end{eqnarray*}
Finally, we compute each portion of  the representation given in Lemma \ref{L1}. For the first integral, we have
\begin{equation*}
\begin{split}
& \int_0^{P_\phi} \biggl [ \vert \psi'(x) \vert^2  - q(x) \vert \psi(x) \vert^2 \biggr ] \ dx\\
& = \int_0^{P_\phi} \biggl [\vert u_0' \vert^4 - 2\lambda_0 \vert u_0  u_0' \vert^2 + (q + \lambda_0)^2 \vert u_0 \vert^4  - 3q\vert u_0 u_0' \vert^2 \biggr ] \ dx\\
\end{split}
\end{equation*}
For the first term of the integral, we integrate by parts and use (\ref{psi0}) to yield
\begin{eqnarray*}
 \int_0^{P_\phi} \vert u_0'(x) \vert^4 \ dx & = &  \int_0^{P_\phi} \left ( \frac{d}{dx} u_0(x) \right ) (u_0'(x))^3 \ dx \\
& = & -3 \int_0^{P_\phi} u_0(x) \vert u_0'(x) \vert^2 u_0''(x) \ dx\\
& = & 3 \int_0^{P_\phi} (q(x) + \lambda_0) \vert u_0(x) u_0'(x) \vert^2 \ dx\\
\end{eqnarray*}
Notice here that we have utilized the conditions $u_0(0) = u_0 \left (P_\phi \right) = 0$  to eliminate boundary terms arising from integration by parts.  Hence, we find
\begin{equation*}
\begin{split}
\int_0^{P_\phi} & \biggl [ \vert \psi'(x) \vert^2  - q(x) \vert \psi(x) \vert^2 \biggr ] \ dx\\
&  = \int_0^{P_\phi} \biggl [ \lambda_0 \vert u_0  u_0' \vert^2 + (q + \lambda_0)^2 \vert u_0 \vert^4 \biggr ] \ dx.
\end{split}
\end{equation*}
The first term of this expression can be simplified further by multiplying (\ref{psi0}) by $u_0^3$ and integrating so that
\begin{eqnarray*}
\int_0^{P_\phi}  (q(x) + \lambda_0) \vert u_0(x) \vert^4 \ dx & = & -\int_0^{P_\phi} u_0''(x) (u_0(x))^3 \ dx\\
& = & \int_0^{P_\phi} u_0'(x) \frac{d}{dx}\left [ (u_0(x))^3 \right ] \ dx\\
& = & 3\int_0^{P_\phi} \vert u_0'(x) u_0(x) \vert^2 \ dx.
\end{eqnarray*}
Using this along with the even symmetry about $\frac{1}{2} P_\phi$ of $\vert u_0 \vert^4$ and $q$, the first piece of the representation satisfies
\begin{equation*}
\begin{split}
\int_0^{P_\phi} & \biggl [ \vert \psi'(x) \vert^2  - q(x) \vert \psi(x) \vert^2 \biggr ] \ dx\\
& =  \int_0^{P_\phi} \biggl [ \frac{1}{3} \lambda_0 (q + \lambda_0) \vert u_0 \vert^4 + (q + \lambda_0)^2 \vert u_0 \vert^4 \biggr ] \ dx\\
& = \int_0^{P_\phi} (q + \lambda_0) \left (q + \frac{4}{3}\lambda_0 \right ) \vert u_0 \vert^4 \ dx
\end{split}
\end{equation*}
Now, since $\psi$ is odd about $\frac{1}{2}P_\phi$, the remaining terms in the representation of Lemma \ref{L1} will vanish. More precisely, for every $e \in [\phi_+,\infty)$ we have
\begin{equation}
\label{vanish1}
\int_{0}^{P_\phi} \frac{\psi(y)}{\sqrt{2(e-\phi(y))}} dy = 0
\end{equation}
and for every $e \in [\phi_-, \phi_+]$ we have
\begin{equation}
\label{vanish2}
\int_{\alpha}^{P_\phi-\alpha} \frac{\psi(y)}{\sqrt{2(e-\phi(y))}} dy = 0.
\end{equation}
To prove (\ref{vanish1}), we first decompose the integral into the reflected components of $\psi$,
\begin{eqnarray*}
\int_{0}^{P_\phi} \frac{\psi(y)}{\sqrt{2(e-\phi(y))}} dy &=& \int_{0}^{\frac{1}{2}P_\phi} \frac{\psi(y)}{\sqrt{2(e-\phi(y))}} dy\\
& \ &  + \int_{\frac{1}{2}P_\phi}^{P_\phi} \frac{\psi(y)}{\sqrt{2(e-\phi(y))}} dy\\
& = & I + II.
\end{eqnarray*}
Now, computing term $II$, we change variables $z = P_\phi - y$ and use the symmetry of $\psi$ and $\phi$ from (\ref{four}) to demonstrate the cancellation of these terms, so that
\begin{eqnarray*}
II &=& \int_{\frac{1}{2}P_\phi}^{P_\phi} \frac{-\psi(P_\phi-y)}{\sqrt{2(e-\phi(y))}} dy\\
&=& -\int_0^{\frac{1}{2}P_\phi} \frac{\psi(z)}{\sqrt{2(e-\phi(P_\phi - z))}} dz\\
&=& -\int_0^{\frac{1}{2}P_\phi} \frac{\psi(z)}{\sqrt{2(e-\phi(z))}} dz\\
&=& -\int_0^{\frac{1}{2}P_\phi} \frac{\psi(z)}{\sqrt{2(e-\phi(z))}} dz\\
& = & - I
\end{eqnarray*}
Hence, adding these terms yields (\ref{vanish1}).  Turning to the justification of (\ref{vanish2}), we again divide the interval of integration about the midpoint $\frac{1}{2} P_\phi$.  Using the symmetry and periodicity of $\phi$ and $u$, we perform the same change of variables as above to find
\begin{equation*}
\begin{split}
\int_{\frac{1}{2}P_\phi}^{P_\phi - \alpha} & \frac{\psi(y)}{\sqrt{2(e-\phi(y))}} dy\\
& = -\int_{\frac{1}{2}P_\phi}^{P_\phi - \alpha}\frac{\psi(P_\phi - y)}{\sqrt{2(e-\phi(y))}} dy\\
& = -\int_\alpha^{\frac{1}{2}P_\phi} \frac{\psi(z)}{\sqrt{2(e-\phi(P_\phi - z))}} dz\\
& = -\int_{\alpha}^{\frac{1}{2}P_\phi} \frac{\psi(z)}{\sqrt{2(e-\phi(z))}} dz.
\end{split}
\end{equation*}
Hence, we conclude (\ref{vanish2}) by adding these terms.

Finally, from these computations and in view of assumption (\ref{six}), we find
\begin{eqnarray*}
\mcL[\psi] & = & \int_0^{P_\phi} \biggl [ \vert \psi'(x) \vert^2 - \int \mu'(e) \ dv \vert \psi(x) \vert^2 \biggr ] \ dx\\
& & \quad + \int_{\phi_+}^\infty \mu'(e) \frac{1}{P_f} \biggl ( \int_0^{P_\phi} \frac{\psi(y)}{\sqrt{2(e - \phi(y))}} \ dy \biggr )^2 \ de\\
& & \quad + 2 \int_{\phi_-}^{\phi_+} \mu'(e) \frac{1}{P_t} \biggl ( \int_{\alpha}^{P_\phi - \alpha} \frac{\psi(y)}{\sqrt{2(e - \phi(y))}} \ dy \biggr )^2 \ de\\
& = & 2\int_0^{\frac{1}{2} P_\phi} (q(x) + \lambda_0) \left (q(x) + \frac{4}{3} \lambda_0 \right ) \vert u_0(x) \vert^4 \ dx\\
&<& 0.
\end{eqnarray*}
which proves the main result.
\end{proof}


Finally, we conclude with the proof of Theorem \ref{T2}.
\begin{proof}[Theorem \ref{T2}]
Let $\mu \in C^1$ and $\phi \in C^2$ satisfying
\begin{equation}
\label{phidp}
\phi''(x) = 1- \int \mu \left (\frac{1}{2} v^2 + \phi(x) \right ) \ dv
\end{equation}
and assumptions (\ref{one}) - (\ref{four}) be given.  We will take advantage of the regularity of these functions.  Assume $q(x) \leq 0$ for all $x \in \left [0,\frac{1}{2} P_\phi \right ].$ Using (\ref{phidp}), we see that the regularity of the right side implies that $\phi \in C^3$ and
$$\phi'''(x) = - q(x) \ \phi'(x).$$
Because $q$ assumes only negative values, we see that $\phi'''(x)$ and $\phi'(x)$ must have the same sign for all $x \in \left [0,\frac{1}{2} P_\phi \right ]$.  From (iv), $\phi$ is strictly decreasing on the interval $[0,\frac{1}{2} P_\phi]$ with a maximum and minimum at $\phi(0)$ and $\phi \left (\frac{1}{2} P_\phi \right )$, respectively.  Thus, $\phi'(0)=\phi'(\frac{1}{2} P_\phi)=0$ and $\phi'(x) < 0$ on the interval $\left (0,\frac{1}{2} P_\phi \right )$.  Therefore, $\phi''' \leq 0$ on $\left (0,\frac{1}{2}P_\phi \right )$ and thus $\phi''$ is decreasing on $\left ( 0, \frac{1}{2} P_\phi \right )$. However, from the above conditions on $\phi'$, we see that $\phi'$ must transition from decreasing to increasing on some subinterval of $\left (0, \frac{1}{2}P_\phi \right )$, and this implies that $\phi''$ must transition from negative values to positive values on the same subinterval.  Clearly such a transition cannot occur if $\phi''(x)$ is decreasing for every $x \in \left ( 0, \frac{1}{2} P_\phi \right )$ and we arrive at a contradiction.  Therefore, there exists $x \in \left [0,\frac{1}{2} P_\phi \right ]$ such that $$q(x) = \int \mu'\left ( \frac{1}{2}v^2 + \phi(x) \right ) \ dv > 0.$$ Additionally, this condition cannot be satisfied for any $\mu$ with $\mu'(e) \leq 0$ for all $e \in [\phi_-,\infty)$ and this completes the proof.

\section{Conclusion}

In summary, specific conditions were determined that guarantee the linear instability of periodic BGK waves for the one-dimensional Vlasov-Poisson system. 
This result then limits the variety of BGK waves which may be stable under singly-periodic perturbations.  
In particular, our main results demonstrate that for a BGK wave to be stable, specific values of a quantity involving the derivative of the particle distribution function cannot be essentially localized within a given interval. 
While this sheds some light on the nature of BGK waves, it remains unknown as to whether they are stable or unstable with respect to general perturbations which share their period. Hence, it seems that the stability properties of these solutions greatly depend on the local behavior of the distribution function and potential.
Additionally, it was shown that modes of the linearized system occur in pairs of opposite sign, thereby displaying that for every perturbation that decays exponentially in time there must be a corresponding perturbation that grows exponentially in time.
Finally, it was shown that the particle distribution function cannot be monotonically decreasing for all energy values, which also limits the stability properties of these waves.  
\end{proof}

\bibliographystyle{acm}
\bibliography{SDPrefs}

\begin{thebibliography}{10}

\bibitem{AshBen}
{\sc Ashbaugh, M., and Benguria, R.}
\newblock Optimal lower bound for the gap between the first two eigenvalues of
  one-dimensional {S}chr\"odinger operators with symmetric single-well
  potentials.
\newblock {\em Proc. Amer. Math. Soc. 105}, 2 (1989), 419--424.

\bibitem{Balmforth}
{\sc Balmforth, N.~J.}
\newblock B{GK} states from the bump-on-tail instability.
\newblock {\em Commun. Nonlinear Sci. Numer. Simul. 17}, 5 (2012), 1989--1997.

\bibitem{BDM}
{\sc Berk, H.~L., Dominguez, R.~R., and Maschke, E.~K.}
\newblock Variational structure of the {V}lasov equation.
\newblock {\em Phys. Fluids 24\/} (1981), 2245.

\bibitem{BGK}
{\sc Bernstein, I.~B., Greene, J.~M., and Kruskal, M.~D.}
\newblock Exact non-linear plasma oscillations.
\newblock {\em Phys. Rev. (2) 108\/} (1957), 546--550.

\bibitem{BP}
{\sc Brewer, D., and Pankavich, S.}
\newblock Computational methods for a one-dimensional plasma model with a
  hyperbolic field.
\newblock {\em SIAM Undergraduate Research Online 4\/} (2011), 81--104.

\bibitem{EFM}
{\sc Engelmann, F., Feix, M., and Minardi, E.}
\newblock Connection between shielding and stability in a collisionless plasma.
\newblock {\em Nuovo Cimento 30\/} (1963), 830.

\bibitem{Fowler}
{\sc Fowler, T.~K.}
\newblock Lyapunov's stability criteria for plasmas.
\newblock {\em J. Mathematical Phys. 4\/} (1963), 559--569.

\bibitem{Glassey}
{\sc Glassey, R.~T.}
\newblock {\em The {C}auchy problem in kinetic theory}.
\newblock Society for Industrial and Applied Mathematics (SIAM), Philadelphia,
  PA, 1996.

\bibitem{GuoStrauss95}
{\sc Guo, Y., and Strauss, W.~A.}
\newblock Instability of periodic {BGK} equilibria.
\newblock {\em Comm. Pure Appl. Math. 48}, 8 (1995), 861--894.

\bibitem{GuoStrauss98}
{\sc Guo, Y., and Strauss, W.~A.}
\newblock Unstable {BGK} solitary waves and collisionless shocks.
\newblock {\em Comm. Math. Phys. 195}, 2 (1998), 267--293.

\bibitem{GuoStrauss99}
{\sc Guo, Y., and Strauss, W.~A.}
\newblock Relativistic unstable periodic {BGK} waves.
\newblock {\em Comput. Appl. Math. 18}, 1 (1999), 87--122.

\bibitem{HagstromMorrison}
{\sc Hagstrom, G.~I., and Morrison, P.~J.}
\newblock On {K}rein-like theorems for noncanonical {H}amiltonian systems with
  continuous spectra: {A}pplication to {V}lasov-{P}oisson.
\newblock {\em Transport Theory Statist. Phys. 39}, 5-7 (2010), 466--501.

\bibitem{Lavine}
{\sc Lavine, R.}
\newblock The eigenvalue gap for one-dimensional convex potentials.
\newblock {\em Proc. Amer. Math. Soc. 121}, 3 (1994), 815--821.

\bibitem{LewisSymon}
{\sc Lewis, H.~R., and Symon, K.~R.}
\newblock Linearized analysis of inhomogeneous plasma equilibria: general
  theory.
\newblock {\em J. Math. Phys. 20}, 3 (1979), 413--436.

\bibitem{Lin}
{\sc Lin, Z.}
\newblock Instability of periodic {BGK} waves.
\newblock {\em Math. Res. Lett. 8}, 4 (2001), 521--534.

\bibitem{LinNonlinear}
{\sc Lin, Z.}
\newblock Nonlinear instability of periodic {BGK} waves for {V}lasov-{P}oisson
  system.
\newblock {\em Comm. Pure Appl. Math. 58}, 4 (2005), 505--528.

\bibitem{ManfrediBertrand}
{\sc Manfredi, G., and Bertrand, P.}
\newblock Stability of {B}ernstein-{G}reene-{K}ruskal modes.
\newblock {\em Phys. Plasmas 7}, 6 (2000), 2425--2431.

\bibitem{MorrisonTTSP}
{\sc Morrison, P.~J.}
\newblock Hamiltonian description of {V}lasov dynamics: action-angle variables
  for the continuous spectrum.
\newblock In {\em Proceedings of the {F}ifth {I}nternational {W}orkshop on
  {M}athematical {A}spects of {F}luid and {P}lasma {D}ynamics ({M}aui, {HI},
  1998)\/} (2000), vol.~29, pp.~397--414.

\bibitem{NP}
{\sc Nguyen, C., and Pankavich, S.}
\newblock A one-dimensional kinetic model of plasma dynamics with a transport
  field.
\newblock {\em Evolution Equations and Control Theory\/} (to appear).

\bibitem{Pfirsch}
{\sc Pfirsch, D.}
\newblock Mikroinstabilitaten vom spiegeltyp in inhomogenen plasmen.
\newblock {\em Z. Naturforsch 17a\/} (1962), 861.

\bibitem{SPW}
{\sc Schindler, K., Pfirsch, D., and Wobig, H.}
\newblock Stability of two-dimensional collision-free plasmas.
\newblock {\em Plasma Phys. 15\/} (1973), 1165.

\bibitem{Schwarz}
{\sc Schwarzmeier, J., Lewis, H.~R., Abraham-Shrauner, B., and Symon, K.~R.}
\newblock Stability of bernstein-greene-kruskal equilibria.
\newblock {\em Phys. Fluids 22\/} (1979), 1747.

\bibitem{VOK}
{\sc Valeo, E., Oberman, C., and Kruskal, M.}
\newblock Property of large-amplitude steady electrostatic waves.
\newblock {\em Phys. Fluids 12\/} (1969), 1246.

\bibitem{VKF}
{\sc van Kampen, N., and Felderhof, B.}
\newblock {\em Theoretical Methods in Plasma Physics}.
\newblock Wiley, New York, NY, 1967.

\bibitem{Zettl}
{\sc Zettl, A.}
\newblock {\em Sturm-{L}iouville theory}, vol.~121 of {\em Mathematical Surveys
  and Monographs}.
\newblock American Mathematical Society, Providence, RI, 2005.

\end{thebibliography}

\end{document}